\documentclass[a4paper,10pt]{amsproc}
\usepackage{amsfonts, amsmath, amssymb, graphicx, mathrsfs}

\usepackage[all]{xy}

\newdir{ >}{!/6pt/@{ }*:(1,0.2)@_{>}*:(1,-0.2)@^{>}}

\theoremstyle{plain}

\newtheorem{theorem}{Theorem}[section]
\newtheorem{lemma}[theorem]{Lemma}

\theoremstyle{definition}
\newtheorem{definition}[theorem]{Definition}

\newtheorem{remark}[theorem]{Remark}
\newtheorem{counter example}[theorem]{Counter Example}
\newtheorem{notation}[theorem]{Notation}

\numberwithin{equation}{section}

\DeclareMathAlphabet{\mathscr}{OT1}{pzc}{m}{it} 
\begin{document}
\Large{
\title{ORDERED FIELD VALUED CONTINUOUS FUNCTIONS WITH COUNTABLE RANGE}

\author[S.K. Acharyya]{Sudip Kumar Acharyya}
\address{Department of Pure Mathematics, University of Calcutta, 35, Ballygunge Circular Road, Kolkata 700019, West Bengal, India}
\email{sdpacharyya@gmail.com}

\author[A. Deb Ray]{Atasi Deb Ray}
\address{Department of Pure Mathematics, University of Calcutta, 35, Ballygunge Circular Road, Kolkata 700019, West Bengal, India}
\email{debrayatasi@gmail.com}

\author[P. Nandi]{Pratip Nandi}
\address{Department of Pure Mathematics, University of Calcutta, 35, Ballygunge Circular Road, Kolkata 700019, West Bengal, India}
\email{pratipnandi10@gmail.com}

\thanks{The third author thanks the CSIR, New Delhi – 110001, India, for financial support}


\begin{abstract}
	For a Hausdorff zero-dimensional topological space $X$ and a totally ordered field $F$ with interval topology, let $C_c(X,F)$ be the ring of all $F-$valued continuous functions on $X$ with countable range. It is proved that if $F$ is either an uncountable field or countable subfield of $\mathbb{R}$, then the structure space of $C_c(X,F)$ is $\beta_0X$, the Banaschewski Compactification of $X$. The ideals $\{O^{p,F}_c:p\in \beta_0X\}$ in $C_c(X,F)$ are introduced as modified countable analogue of the ideals $\{O^p:p\in\beta X\}$ in $C(X)$. It is realized that $C_c(X,F)\cap C_K(X,F)=\bigcap_{p\in\beta_0X\texttt{\textbackslash}X} O^{p,F}_c$, this may be called a countable analogue of the well-known formula $C_K(X)=\bigcap_{p\in\beta X\texttt{\textbackslash}X}O^p$ in $C(X)$. Furthermore, it is shown that the hypothesis $C_c(X,F)$ is a Von-Neumann regular ring is equivalent to amongst others the condition that $X$ is a $P-$space.
\end{abstract}
\thanks {}
\maketitle
\section{Introduction}
Let $F$ be a totally ordered field equipped with its ordered topology. For any topological space $X$, suppose $C(X,F)$ is the set of all $F-$valued continuous functions on $X$. This later set becomes a commutative lattice ordered ring with unity, if the operations are defined pointwise on $X$. As in classical scenario with $F=\mathbb{R}$, there is already discovered an interplay existing between the topological structure of $X$ and the algebraic ring and order structure of $C(X,F)$ and a few of its chosen subrings. In order to study this interaction, one can stick to a well-chosen class of spaces viz. the so-called completely $F-$regular topological spaces or in brief $CFR$ spaces. $X$ is called $CFR$ space if it is Hausdorff and points and closed sets in $X$ could be separated by $F-$valued continuous functions in an obvious manner. Problems of this kind are addressed in \cite{IJMS2004}, \cite{AGT2004}, \cite{SABM2004}, \cite{TP2016}, \cite{AGT2015}, \cite{TAMS1975}. It turns out that with $F\neq\mathbb{R}$, $CFR$ spaces are precisely zero-dimensional spaces. Thus zero-dimensionality on $X$ can be realized as a kind of separation axiom effected by $F-$valued continuous functions on $X$. In the present article, we intend to examine the countable analogue of the ring $C(X,F)$ vis-a-vis the corresponding class of spaces $X$. Towards that end, we let $C_c(X,F)=\{f\in C(X,F): f(X)~is~a~countable~subset~of~F\}$. Then $C_c(X,F)$ is a subring as well as a sublattice of $C(X,F)$. It is interesting to note that spaces $X$ in which points and closed sets can be separated by functions in $C_c(X,F)$ are exactly zero-dimensional also [Theorem $2.9$]. Furthermore, the set of all maximal ideals in $C_c(X,F)$ endowed with the well-known Hull-Kernel topology (also known as the structure space of $C_c(X,F)$) turns out to be homeomorphic to the Banaschewski Compactification $\beta_0X$ of $X$ [Theorem $2.18$]. To achieve this result, we have to put certain restriction on the nature of the totally ordered field $F$ viz. that $F$ is either an uncountable field or a countable subfield of $\mathbb{R}$. A special case of this result choosing $F=\mathbb{R}$ reads: the structure space of the ring $C_C(X)$ consisting all real-valued continuous functions on $X$ with countable range is $\beta_0X$, which is Remark $3.6$ in \cite{RM2018}. Since the maximal ideals of $C_c(X,F)$ can be indexed by virtue of the points of $\beta_0X$, it is not surprising that a complete description of these ideals can be given by the family $\{M^{p,F}_c:p\in\beta_0X\}$, where $M^{p,F}_c=\{f\in C_c(X,F):p\in cl_{\beta_0X}Z_c(f)\}$, here $Z_c(f)=\{x\in X:f(x)=0\}$ stands for the zero set of $f$ [Remark $2.21$]. This is analogous to the Gelfand-Kolmogorov Theorem $7.3$ \cite{GJ}. Also, this places Theorem $4.8$ \cite{RS2013} on a wider setting. As a natural companion of $M^{p,F}_c$, we introduce the ideal $O^{p,F}_c=\{f\in M^{p,F}_c:cl_{\beta_0X}Z_c(f)~is~a~neighbourhood~of~p~in~\beta_0X\}$. Amongst other facts connecting these two classes of ideals in $C_c(X,F)$, we have realized that the ideals that lie between $O^{p,F}_c$ and $M^{p,F}_c$ are precisely those that extend to unique maximal ideals in $C_c(X,F)$ [Theorem $3.1(4)$]. This may be called the modified countable counterpart of Theorem $7.13$ in \cite{GJ}. Also see Lemma $4.11$ in \cite{RM2018} in this connection. If $C^c_K(X,F)=\{f\in C_c(X,F):cl_X(X-Z_c(f))~is~compact\}$, then we have found out a formula for this ring in terms of the ideals $O^{p,F}_c$ as follows : $C^c_K(X,F)=\bigcap_{p\in\beta_0X\texttt{\textbackslash}X}O^{p,F}_c$ [in Theorem $3.5$, compare with the Theorem $3.9$, \cite{TP2016}]. This we may call the appropriate modified countable analogue of the well-known formula in $C(X)$ which says that $C_K(X)=\bigcap_{p\in\beta X\texttt{\textbackslash}X}O^p$ [$7E$, \cite{GJ}]. The above-mentioned results constitute technical section $2$ and $3$ of the present article.\\
In the final section $4$ of this article, we have examined several possible consequences of the hypothesis that $C_c(X,F)$ is a Von-Neumann regular ring with $F$, either an uncountable field or a countable subfield of $\mathbb{R}$. To aid to this examination, we introduce $\mathscr{m}^F_c-$topology on $C_c(X,F)$ as a modified version of $\mathscr{m}_c-$topology on $C_c(X)$ already introduced in \cite{QM2020}. We establish amongst a host of necessary and sufficient conditions that $C_c(X,F)$ is a Von-Neumann regular ring if and only if each ideal in $C_c(X,F)$ is closed in the $\mathscr{m}^F_c-$topology if and only if $X$ is a $P-$space. This places theorem $3.9$ in \cite{QM2020} on a wider settings, and we may call it a modified countable analogue of the well-known fact that $X$ is a $P-$space when and only when each ideal in $C(X)$ is closed in the $\mathscr{m}-$topology [$7Q4$, \cite{GJ}].
\section{Duality between ideals in $C_c(X,F)$ and $Z_{F_c}-$filters on $X$}
\begin{notation}
	In spite of the difference of notations, we write for $f\in C_c(X,F)$, $Z_c(f)\equiv\{x\in X: f(x)=0\}\equiv Z(f)$
\end{notation}
Let $Z_c(X,F)=\{Z_c(f):f\in C_c(X,F)\}$.\\
An ideal unmodified in a ring will always stand for a proper ideal.
\begin{definition}
	A filter of zero sets in the family $Z_c(X,F)$ is called a $Z_{F_c}-$filter on $X$. A $Z_{F_c}-$filter on $X$ is called a $Z_{F_c}-$ultrafilter on $X$ if it is not properly contained in any $Z_{F_c}-$filter on $X$.
\end{definition}
\begin{remark}
	A straight forward use of Zorn's Lemma tells that a $Z_{F_c}-$filter on $X$ extends to a $Z_{F_c}-$ultrafilter on $X$. Furthermore any subfamily of $Z_c(X,F)$ with finite intersection property can be extended to a $Z_{F_c}-$ultrafilter on $X$.
\end{remark}
The following results correlating $Z_{F_c}-$filters on $X$ and ideals in $C_c(X,F)$ can be established by using routine arguments.
\begin{theorem}\hspace{-9mm}
	\begin{enumerate}
		\item If $I$ is an ideal in $C_c(X,F)$, then $Z_{F,C}[I]=\{Z_c(f):f\in I\}$ is a $Z_{F_c}-$filter on $X$. Dually for a $Z_{F_c}-$filter $\mathscr{F}$ on $X$, $Z^{-1}_{F,C}[\mathscr{F}]=\{f\in C_c(X,F):Z_c(f)\in \mathscr{F}\}$ is an ideal in $C_c(X,F)$.
		\item If $M$ is a maximal ideal in $C_c(X,F)$, then $Z_{F,C}[M]$ is a $Z_{F_c}-$\\ultrafilter on $X$. If $\mathscr{U}$ is a $Z_{F_c}-$ultrafilter on $X$, then $Z^{-1}_{F,C}[\mathscr{U}]$ is a maximal ideal in $C_c(X,F)$.
	\end{enumerate}
\end{theorem}
\begin{definition}
	An ideal $I$ in $C_c(X,F)$ is called $Z_{F_c}-$ideal if $Z^{-1}_{F,C}[Z_{F,C}[I]]=I$
\end{definition}
It follows from Theorem $2.4(2)$ that each maximal ideal in $C_c(X,F)$ is a $Z_{F_c}-$ideal. Hence the assignment : $M\to Z_{F,C}[M]$ establish a one-to-one correspondence between the maximal ideals in $C_c(X,F)$ and the $Z_{F_c}-$ultrafilters on $X$.\\
The following propositions can be easily established on using the arguments adopted in Chapter $2$ and Chapter $4$ of \cite{GJ} in a straight forward manner.
\begin{theorem}
	A $Z_{F_c}-$ideal $I$ in $C_c(X,F)$ is a prime ideal if and only if it contains a prime ideal. Hence each prime ideal in $C_c(X,F)$ extends to a unique maximal ideal, in other words, $C_c(X,F)$ is a Gelfand ring.
\end{theorem}
\begin{theorem}
	The complete list of fixed maximal ideals in $C_c(X,F)$ is given by $\{M^c_{p,F}:p\in X\}$ where $M^c_{p,F}=\{f\in C_c(X,F):f(p)=0\}$. An ideal $I$ in $C_c(X,F)$ is called fixed if $\bigcap_{f\in I}Z(f)\neq\phi$.
\end{theorem}
\begin{definition}
	$X$ is called countably completely $F-$regular or in brief $CCFR$ space if it is Hausdorff and given a closed set $K$ in $X$ and a point $x\in X\texttt{\textbackslash}K$, there exists $f\in C_c(X,F)$ such that $f(x)=0$ and $f(K)=1$.
\end{definition}
It is clear that a $CCFR$ space is $CFR$.\\
A $CFR$ space with $F\neq\mathbb{R}$ is zero-dimensional by Theorem $2.3$ in \cite{TP2016}. A $CCFR$ space with $F=\mathbb{R}$ is the same as $C-$completely regular space introduced in \cite{RS2013} and is hence zero-dimensional space by Proposition $4.4$ in \cite{RS2013}. Thus for all choices of the field $F$, a $CCFR$ space is zero-dimensional. Conversely, it is easy to prove that a zero-dimensional space $X$ is $CCFR$ for any totally ordered field $F$. Thus, the following result comes out immediately.
\begin{theorem}
	The statements written below are equivalent for a Hausdorff space $X$ and for any totally ordered field $F$ : 
	\begin{enumerate}
		\item $X$ is zero-dimensional.
		\item $X$ is $CCFR$.
		\item $Z_c(X,F)$ is a base for closed sets in $X$.
	\end{enumerate}
\end{theorem}
The following result tells that as in the classical situation with $F=\mathbb{R}$, in the study of the ring $C_c(X,F)$, one can assume without loss of generality that the ambient space $X$ is $CCFR$, i.e., zero-dimensional.
\begin{theorem}
	Let $X$ be a topological space and $F$, a totally ordered field. Then it is possible to construct a zero-dimensional Hausdorff space $Y$ such that the ring $C_c(X,F)$ is isomorphic to the ring $C_c(Y,F)$
\end{theorem}
We need the following two subsidiary results to prove this theorem. 
\begin{lemma}
	A Hausdorff space $X$ is zero-dimensional if and only if given any ordered field $F$, there exists a subfamily $\mathscr{S}\subset F^X_c=\{f\in F^X:f(X)~is~countable~set\}$, which determines the topology on $X$ in the sense that, the given topology on $X$ is the smallest one with respect to which each function in $\mathscr{S}$ is continuous.
\end{lemma}
The proof of this lemma can be accomplished by closely following the arguments in Theorem $3.7$ in \cite{GJ} and using Theorem $2.9$.\\
\begin{lemma}
	Suppose $X$ is a topological space whose topology is determined by a subfamily $\mathscr{S}$ of $F^X_c$. Then for a topological space $Y$, a function $h:Y\to X$ is continuous if and only if for each $g\in\mathscr{S}$, $g\circ h:Y\to F$ is a continuous map.
\end{lemma}
The proof of the last lemma is analogous to that of Theorem $3.8$ in \cite{GJ}.
\begin{proof}
 of the main theorem :  Define a binary relation $\lq{\sim}$' on $X$ as follows : for $x,y\in X$, $x\sim y$ if and only if for each $f\in C_c(X,F),~f(x)=f(y)$.\\
Suppose $Y=\{[x]:x\in X\}$, the set of all corresponding disjoint classes. Let $\tau:X\to Y$ be the canonical map given by $\tau(x)=[x]$. Each $f\in C_c(X,F)$ gives rise to a function $g_f:Y\to F$ as follows : $g_f[x]=f(x)$.\\
Let $\mathscr{S}=\{g_f:f\in C_c(X,F)\}$. Then $\mathscr{S}\subset F^Y_c$. Equip $Y$ with the smallest topology, which makes each function in $\mathscr{S}$ continuous. It follows from the Lemma $2.11$, that $Y$ is a zero-dimensional space and it is easy to check that $Y$ is Hausdorff. The continuity of $\tau$ follows from Lemma $2.12$. Now by the following arguments in Theorem $3.9$ in \cite{GJ}, we can prove that the assignment : $C_c(Y,F)\to C_c(X,F)$ : $g\to g\circ\tau$ is an isomorphism onto $C_c(X,F)$.
\end{proof}
The following result is a countable counterpart of a portion of Theorem $4.11$ in \cite{GJ}.
\begin{theorem}
	For a zero-dimensional Hausdorff space $X$ and a totally ordered field $F$, the following three statements are equivalent : 
	\begin{enumerate}
		\item $X$ is compact.
		\item Each ideal in $C_c(X,F)$ is fixed.
		\item Each maximal ideal in $C_c(X,F)$ is fixed.
	\end{enumerate}
\end{theorem}
\begin{proof}
	 $(1)\implies(2)$ and $(2)\implies(3)$ are trivial. We prove $(3)\implies(1)$ : Let $(3)$ be true.\\
	Suppose $\mathscr{B}$ is a subfamily of $Z_c(X,F)$ with finite intersection property. Since $Z_c(X,F)$ is a base for the closed sets in $X$(vide Theorem $2.9$), it suffices to show that $\bigcap\mathscr{B}\neq\phi$.\\
	Indeed $\mathscr{B}$ can be extended to a $Z_{F_c}-$ultrafilter $\mathscr{U}$ on $X$. In view of Theorem $2.4$, we can write $ \mathscr{U}=Z_{F,C}[M]$ for a maximal ideal $M$ in $C_c(X,F)$. Hence $\bigcap\mathscr{B}\supset \bigcap\mathscr{U}\neq\phi$.
\end{proof}
Before proceeding further, we reproduce below the following basic facts about the structure space of a commutative ring with unity from $7M$, \cite{GJ}.\\
Let $A$ be a  commutative ring with unity and $\mathcal{M}(A)$, the set of all maximal ideals in $A$. For each $a\in A$, let $\mathcal{M}_a=\{M\in\mathcal{M}(A):a\in M\}$. Then the family $\{\mathcal{M}_a:a\in A\}$ constitutes a base for the closed sets of some topology $\tau$ on $\mathcal{M}(A)$. The topological space $(\mathcal{M}(A),\tau)$ is known as the structure space of $A$ and is a compact $T_1$ space. If $A$ is a Gelfand ring, then it is established in Theorem $1.2$, \cite{AMS1971} that $\tau$ is a Hausdorff topology on $\mathcal{M}(A)$. The closure of a subset $\mathcal{M}_0$ of $\mathcal{M}(A)$ is given by : $\overline{\mathcal{M}_0}=\{M\in\mathcal{M}(A):M\supset\bigcap\mathcal{M}_0\}\equiv$ the hull of the kernel of $\mathcal{M}_0$. [This is the reason why $\tau$ is also called the hull-kernel topology on $\mathcal{M}(A)$].\\
Let us denote the structure space of the ring $C_c(X,F)$ by the notation $\mathcal{M}_c(X,F)$. Since $C_c(X,F)$ is a Gelfand ring, already verified in Theorem $2.6$, it follows that $\mathcal{M}_c(X,F)$ is a compact Hausdorff space. From now on, we assume that $X$ is Hausdorff and zero-dimensional, and we will stick to this hypothesis throughout this article. It follows that the assignment $\psi:X\to \mathcal{M}_c(X,F)$ given by $\psi(p)=M^c_{p,F}$ is one-to-one. Furthermore for any $f\in C_c(X,F)$, 
$$\psi(Z_c(f))=\{M^c_{p,F}:f\in M^c_{p,F}\}=\mathcal{M}_f\cap \psi(X)$$
where $\mathcal{M}_f=\{M\in\mathcal{M}_c(X,F): f\in M\}$.\\
This shows that $\psi$ exchanges the basic closed sets of the two spaces $X$ and $\psi(X)$. Finally, 
\begin{align*}
\overline{\psi(X)}&=\{M\in\mathcal{M}_c(X,F):M\supset\bigcap\psi(X)\}\\
&=\{M\in\mathcal{M}_c(X,F):M\supset\bigcap_{p\in X}\{M^c_{p,F}\}=\{0\}\}\\
&=\mathcal{M}_c(X,F)
\end{align*}
This leads to the following proposition : 
\begin{theorem}
	The map $\psi:X\to \mathcal{M}_c(X,F)$ given by $\psi(p)=M^c_{p,F}$ defines a topological embedding of $X$ onto a dense subspace of $\mathcal{M}_c(X,F)$. In a more formal language, the pair $(\psi,\mathcal{M}_c(X,F))$ is a Hausdorff Compactification of $X$.
\end{theorem}
The next result shows that the last-mentioned compactification enjoys a special extension property.
\begin{theorem}
	The compactification $(\psi,\mathcal{M}_c(X,F))$ enjoys the $C-$extension property (see Definition $2.5$ in \cite{QM2020}) in the following sense, given a compact Hausdorff zero-dimensional space $Y$ and a continuous map $f:X\to Y$, there can be defined a continuous map $f^c:\mathcal{M}_c(X,F)\to Y$ with the following property : $f^c\circ\psi=f$.
\end{theorem}
\begin{proof}
	This can be accomplished by closely adapting the arguments made in the second paragraph in the proof of the Theorem $2.7$ in \cite{QM2020}. However, to make the paper self-contained, we sketch a brief outline of the main points of its proof.\\
	Let $M\in\mathcal{M}_c(X,F)$. Define as in \cite{QM2020}, $\widetilde{M}=\{g\in C_c(Y,F):g\circ f\in M\}$. Then $\widetilde{M}$ is a prime ideal in $C_c(Y,F)$. Since $C_c(Y,F)$ is Gelfand ring and $Y$ is compact and zero-dimensional, it follows from Theorem $2.13$ that there exists a unique $y\in Y$ such that $\bigcap_{g\in\widetilde{M}}Z_c(g)=\{y\}$. Set $f^c(M)=y$. Then $f^c:\mathcal{M}_c(X,F)\to Y$ is the desired continuous map.
\end{proof}
\begin{remark}
	If the structure space $\mathcal{M}_c(X,F)$ of $C_c(X,F)$ is zero-dimensional, then $(\psi, \mathcal{M}_c(X,F))$ is topologically equivalent to the Banaschewski Compactification $\beta_0X$ of $X$. [see the comments after Definition $2.5$ in \cite{QM2020}].
\end{remark}
We shall now impose a condition on $F$; sufficient to make $\mathcal{M}_c(X,F)$ zero-dimensional.
\begin{theorem}
	Suppose the totally ordered field $F$ is either uncountable or a countable subfield of $\mathbb{R}$. Then given $f\in C_c(X,F)$, there exists an idempotent $e\in C_c(X,F)$ such that $e$ is a multiple of $f$ and $(1-e)$ is a multiple of $(1-f)$
\end{theorem}
\begin{proof}
	We prove this theorem with the assumption that $F$ is uncountable. The proof for the case when $F$ is a countable subfield of $\mathbb{R}$ can be accomplished on using some analogous arguments. We first assert that the interval $[0,1]=\{\alpha\in F: 0\leq\alpha\leq 1\}$ is an uncountable set. This is immediate if $F$ is Archimedean ordered because in that case $F^+=\{\alpha\in F: \alpha
	\geq 0\}=\bigcup_{n\in\mathbb{N}\cup \{0\}}[n,n+1]$ and for each $n\in\mathbb{N}\cup\{0\}$, $[n,n+1]$ is equipotent with $[0,1]$ through the translation map : $\alpha\to(\alpha+n),~\alpha\in[0,1]$. Now suppose that $F$ is non-Archimedean ordered field. If possible let $[0,1]$ be a countable set. Then the set $F^+\texttt{\textbackslash}\bigcup_{n\in\mathbb{N}\cup \{0\}}[n,n+1]$ becomes an uncountable set, which  means that the set of all infinitely large members of $F$ make an uncountable set. Consequently, the set $I=\{\alpha\in F^+:0<\alpha<\frac{1}{n}~for~each~n\in\mathbb{N}\}$ comprising of the infinitely small members of $F$ is an uncountable set. But it is easy to see that $I\subset (0,1)$ and therefore $(0,1)$ turns out to be an uncountable set -- a contradiction. Thus it is proved that $[0,1]$ is an uncountable set [and consequently for any $\alpha>0$ in $F$, $(0,\alpha)$ becomes an uncountable set]. So we can choose $r\in (0,1)$ such that $r\notin f(X)$. Let, $W=\{x\in X:f(x)<r\}=\{x\in X:f(x)\leq r\}$ and so $X\texttt{\textbackslash}W=\{x\in X:f(x)>r\}=\{x\in X: f(x)\geq r\}$. It is clear that $W$ and $X\texttt{\textbackslash}W$ are clopen sets in $X$ and the function $e:X\to F$ defined by $e(W)=\{0\}$ and $e(X\texttt{\textbackslash}W)=\{1\}$ is an idempotent in $C_c(X,F)$. We see that $Z_c(f)\subset Z_c(e)$ and $Z_c(1-f)\subset Z_c(1-e)$ and we can say that $Z_c(e)$ is a neighbourhood of $Z_c(f)$ and $Z_c(1-e)$ is a neighbourhood of $Z_c(1-f)$ in the space $X$. Hence $e$ is a multiple of $f$ and $(1-e)$ is a multiple of $(1-f)$. [compare with the arguments made in Remark $3.6$ in \cite{RM2018}].
\end{proof}
\begin{theorem}
	The structure space $\mathcal{M}_c(X,F)$ of $C_c(X,F)$ is zero-dimensional and hence $\mathcal{M}_c(X,F)=\beta_0X$.
\end{theorem}\vspace{-3mm}
[Here $F$ is either uncountable or a countable subfield of $\mathbb{R}$]
\begin{proof}
	Recall the notation for $f\in C_c(X,F)$, $\mathcal{M}_f=\{M\in\mathcal{M}_c(X,\\F):f\in M\}$. Suppose $M\in\mathcal{M}_c(X,F)$ is such that $M\notin\mathcal{M}_f$. It suffices to find out an idempotent $e$ in $C_c(X,F)$ with the property : $\mathcal{M}_f\subset\mathcal{M}_e$ and $M\notin\mathcal{M}_e$. The simple reason is that $e.(1-e)=e-e^2=e-e=0$ and hence $\mathcal{M}_e=\mathcal{M}_c(X,F)\texttt{\textbackslash}\mathcal{M}_{(1-e)}$, consequently $\mathcal{M}_e$ is a clopen set in $\mathcal{M}_c(X,F)$. Now towards finding out such an idempotent let us observe that $M\notin\mathcal{M}_f$ implies that  $f\notin M$, which further implies that $<f,M>=C_c(X,F)$. Hence we can write  : $1=f.h+g$, where $h\in C_c(X,F)$ and $g\in M$. By Theorem $2.17$, there exists an idempotent $e$ in $C_c(X,F)$ such that $e=g_1.g$ and $(1-e)=g_2.(1-g)$, where $g_1,g_2\in C_c(X,F)$. Now let $N\in\mathcal{M}_f$, then $f\in N$ and so $f.h\in N$, which implies that $(1-g)\in N$ consequently $(1-e)\in N$. Therefore $e\notin N$, which means that $N\notin\mathcal{M}_e$, i.e., $N\in\mathcal{M}_c(X,F)\texttt{\textbackslash}\mathcal{M}_e$. Again since $g\in M$, it follows that $e\in M$, thus $M\in \mathcal{M}_e$.
\end{proof}
\begin{remark}
	On choosing $F=\mathbb{R}$ and $X=\mathbb{Q}$ in the above Theorem $2.18$, we get that $\beta_0\mathbb{Q}=$ structure space of $C(\mathbb{Q},\mathbb{R})=\beta\mathbb{Q}$. Thus $\beta\mathbb{Q}$ becomes zero-dimensional, i.e., $\mathbb{Q}$ is strongly zero-dimensional. This is a standard result in General Topology -- indeed a Lindel\"{o}f zero-dimensional space is strongly zero-dimensional. [Theorem $6.2.7$, \cite{GT}].
\end{remark}
One of the major achievements in the theory of $C(X)$ is that a complete description of the maximal ideals in this ring can be given. This is a remark made in the beginning of Chapter $6$ in \cite{GJ}. In order to give such a description, it becomes convenient to archive $\beta X$ as the space of $Z-$ultrafilter on $X$ equipped with the Stone- topology and formal construction of such a thing is dealt in rigorously in Chapter $6$ in \cite{GJ}. We follow the same technique in order to furnish an explicit description of maximal ideals in $C_c(X,F)$.\\
For each $p\in X$, let $A^c_{p,F}=\{Z\in Z_c(X,F):p\in Z\}\equiv Z_{F,C}[M^c_{p,F}]$. Thus $X$ is a readymade index set for the family of fixed $Z_{F_c}-$ultrafilters on $X$. As in Chapter $6$, \cite{GJ}, we extend the set $X$ to a set $\alpha X$ to serve as  an index set for the family of all $Z_{F_c}-$ultrafilters on $X$. For $p\in\alpha X$, let the corresponding $Z_{F_c}-$ultrafilter be designated as $A^{p,F}_c$ with the understanding that if $p\in X$, then $A^{p,F}_c=A^c_{p,F}$.\\
For $Z\in Z_c(X,F)$, let $\overline{Z}=\{p\in\alpha X:Z\in A^{p,F}_c\}$. Then $\{\overline{Z}:z\in Z_c(X,F)\}$ makes a base for the closed sets of some topology on $\alpha X$ in which for $Z\in Z_c(X,F)$, $\overline{Z}=cl_{\alpha X}Z$. Furthermore, for $Z_1,Z_2\in Z_c(X,F)$, $\overline{Z_1\cap Z_2}=\overline{Z_1}\cap\overline{Z_2}$ and $\alpha X$ becomes a compact Hausdorff space containing $X$ as a dense subset. Also given a point $p\in\alpha X$, $A^{p,F}_c$ is the unique $Z_{F_c}-$ultrafilter on $X$ which converges to $p$ and finally $\alpha X$ possesses the $C-$extension property meaning that if $Y$ is a compact Hausdorff zero-dimensional space and $f:X\to Y$, a continuous map, then $f$ can be extended to a continuous map $f^{\#}:\alpha X\to Y$. All these facts can be realized just by closely following the arguments in Chapter $6$ in \cite{GJ}.
\begin{theorem}
	The space $\alpha X$ is a zero-dimensional space. 
\end{theorem}\vspace{-3mm}
[Blanket assumption: $F$ is either an uncountable field or a countable subfield of $\mathbb{R}$]
\begin{proof}
	Let $p\in\alpha X$ and $Z\in Z_c(X,F)$ be such that $p\notin\overline{Z}$.\\
	It suffices to find out a clopen set $K$ in $\alpha X$ such that $\overline{Z}\subset K$ and $p\notin K$.\\
	Now $p\notin\overline{Z}\implies Z\notin A^{p,F}_c$. Since $A^{p,F}_c$ is a $Z_{F_c}-$ultrafilter, this implies that there exists $Z^*\in A^{p,F}_c$ such that $Z\cap Z^*=\phi$. Hence there exists $f\in C_c(X,F)$ such that $f:X\to [0,1]$ in $F$ such that $f(Z^*)=\{0\}$ and $f(Z)=\{1\}$. Using the hypothesis that $F$ is an uncountable field, and take note of the arguments in the proof of the Theorem $2.17$, we can find out an $r\in(0,1)$ in $F$ such that $r\notin f(X)$ [analogous arguments can be made if $F$ is a countable subfield of $\mathbb{R}$].\\
	Let $K=\{x\in X:f(x)>r\}=\{x\in X:f(x)\geq r\}$. Then $K$ is a clopen set in $X$ containing $Z$ and therefore $\overline{Z}\subset K$. Now $Z^*\subset X\texttt{\textbackslash}K$ implies $Z^*\cap K=\phi$ and hence $\overline{Z^*}\cap\overline{K}=\phi$, i.e., $\overline{Z^*}\cap K=\phi$. Since $Z^*\in A^{p,F}_c$ and therefore $p\in Z^*$, this further implies that $p\notin K$.
\end{proof}
\begin{remark}
	Since $\alpha X$ enjoys the $C-$extension property and is zero-dimensional, it follows from Definition $2.5$ in \cite{QM2020} that $\alpha X$ is essentially the same as $\beta_0 X$, the Banaschewski Compactification of $X$ and hence we can write for any $p\in\beta_0 X$ and $Z\in Z_c(X,F)$, $Z\in A^{p,F}_c$ if and only if $p\in cl_{\beta_0 X}$. If we now write  $M^{p,F}_c=Z_{F,C}^{-1}[A^{p,F}_c]$, then this becomes a maximal ideal in $C_c(X,F)$. Since by Theorem $2.4(2)$, there is already realized a one-to-one correspondence between maximal ideals in $C_c(X,F)$ and $Z_{F_c}-$ultrafilters on $X$ via the map $M\to Z_{F,C}[M]$, a complete description of the maximal ideals in $C_c(X,F)$ is given by the list $\{M^{p,F}_c:p\in\beta_0 X\}$ where $M^{p,F}_c=\{f\in c_c(X,F):p\in cl_{\beta_0 X}Z_c(f)\}$
\end{remark}
\section{The ideals $O^{p,F}_c$ and a formula for $C^c_K(X,F)$}
For each $p\in\beta_0 X$, set $$O^{p,F}_c=\{f\in C_c(X,F):cl_{\beta_0 X}Z_c(f)~is~a~neighbourhood~of~p~in~\beta_0 X\}$$
Then the following facts come out as modified countable analogue of the relations between the ideals $M^p$ and $O^p$ in the classical scenario recorded in $7.12$, $7.13$, $7.15$ in \cite{GJ}. Also see Lemma $4.11$ in \cite{RM2018} in this connection.
\begin{theorem}
	Let the ordered field $F$ be either uncountable or a countable subfield of $\mathbb{R}$. Then for a zero-dimensional Hausdorff space $X$, the following statements are true : 
	\begin{enumerate}
		\item $O^{p,F}_c$ is a $Z_{F_c}-$ideal in $C_c(X,F)$ contained in $M^{p,F}_c$.
		\item $O^{p,F}_c=\{f\in C_c(X,F):there~exists~an~open~neighbourhood~V\\of~p~in~\beta_0 X~such~that~Z_c(f)\supset V\cap X\}$.
		\item For $p\in\beta_0X$ and $f\in C_c(X,F)$, $f\in O^{p,F}_c$ if and only if there exists $g\in C_c(X,F)\texttt{\textbackslash}M^{p,F}_c$ such that $f.g=0$, hence each non-zero element in $O^{p,F}_c$ is a divisor of zero in $C_c(X,F)$. Indeed $O^{p,F}_c$ is a $z^o-$ideal in $C_c(X,F)$.
		\item An ideal $I$ in $C_c(X,F)$ is extendable to a unique maximal ideal if and only if there exists $p\in\beta_0 X$ such that $O^{p,F}_c\subset I$.
		\item For $p\in\beta_0X$, $O^{p,F}_c$ is a fixed ideal if and only if $p\in X$.
	\end{enumerate}
\end{theorem}
\begin{proof}
	The statements $(1)$, $(2)$ and $(4)$ can be proved by making arguments parallel to those adopted to prove the corresponding results in the classical situation with $F=\mathbb{R}$ in Sections $7.12,7.13, 7.15$ in \cite{GJ}. We prove only the statements $(3)$ and $(5)$.\\
	To prove $(3)$, let $f\in O^{p,F}_c$. Then by $(2)$, there exists an open neighbourhood $V$ of $p$ in $\beta_0X$ such that $Z_c(f)\supset V\cap X$. Since $\beta_0X$ is zero dimensional, there exists a clopen set $K$ in $\beta_0X$ such that $\beta_0X\texttt{\textbackslash}V\subset K$ and $p\notin K$. The function $h:\beta_0X\to F$, defined by $h(K)=\{0\}$ and $h(\beta_0X\texttt{\textbackslash}K)=\{1\}$ belongs to $C_c(\beta_0X,F)$. Take $g=h|_X$. Then $g\in C_c(X,F)$, $f.g=0$ and $p\notin cl_{\beta_0X}Z_c(g)$, hence $g\notin M^{p,F}_c$.\\
	Conversely let there exist $g\in C_c(X,F)\texttt{\textbackslash}M^{p,F}_c$ such that $f.g=0$. Then $p\notin cl_{\beta_0X}Z_c(g)$. Therefore there exists an open neighbourhood $V$ of $p$ in $\beta_0X$ such that $V\cap Z_c(g)=\phi$. Since $Z_c(f)\cup Z_c(g)=X$, it follows that $X\cap V\subset Z_c(f)$. Hence from $(2)$, we get that $f\in O^{p,F}_c$.\\
	To prove the last part of $(3)$, we recall that an ideal $I$ in a commutative ring $A$ with unity is called a $z^o-$ideal if for each $a\in I$, $P_a\subset I$, where $P_a$ is the intersection of all minimal prime ideals in $A$ containing $a$. We reproduce the following useful formula from Proposition $1.5$ in \cite{CA2000}, which is also recorded in Theorem $3.10$ in \cite{QM2020} : if $A$ is a reduced ring meaning that $0$ is the only nilpotent member of $A$, then $P_a=\{b\in A: Ann(a)\subset Ann(b)\}$, where $Ann(a)=\{c\in A:a.c=0\}$ is the annihilator of $a$ in $A$. Hence for any $f\in C_c(X,F)$, $P_f\equiv$ the intersection of all minimal prime ideals in $C_c(X,F)$ which contain $f=\{g\in C_c(X,F): Ann(f)\subset Ann(g)\}$.\\
	Now to show that $O^{p,F}_c$ is a $z^o-$ideal in $C_c(X,F)$, for any $p\in\beta_0X$, choose $f\in O^{p,F}_c$ and $g\in P_f$. Therefore $Ann(f)\subset Ann(g)$. But from the result $(3)$, we see that there exists $h\in C_c(X,F)\texttt{\textbackslash}M^{p,F}_c$ such that $f.h=0$ and hence $h\in Ann(f)$. Consequently, $h\in Ann(g)$, i.e., $g.h=0$. Thus $P_f\subset O^{p,F}_c$ and hence $O^{p,F}_c$ is a $z^o-$ideal in $C_c(X,F)$.\\
	Proof of $(5)$ : If $p\in X$, then $M^{p,F}_c=M^c_{p,F}$, a fixed ideal, hence $O^{p,F}_c$ is also fixed.\\
	Now let $p\in\beta_0X\texttt{\textbackslash}X$. Choose $x\in X$ and a closed neighbourhood $W$ of $p$ in $\beta_0X$ such that $x\notin W$. Since $\beta_0X$ is zero-dimensional, there exists a clopen set $K$ in $\beta_0X$ such that $W\subset K$ and $x\notin K$. Let $g:\beta_0X\to F$ be defined by $g(K)=\{0\}$ and $g(\beta_0X\texttt{\textbackslash}K)=\{1\}$. Then $g\in C_c(\beta_0X,F)$ and hence $h=g|_X\in C_c(X,F)$. We observe that $h(x)=1$ and $Z_c(h)\supset K\cap X$. It follows from the result $(2)$ that $h\in O^{p,F}_c$. This proves that $O^{p,F}_c$ is a free ideal in $C_c(X,F)$
\end{proof}
The following properties of $C^c_K(X,F)=\{f\in C_c(X,F):f~has\\compact~support~i.e.~cl_X(X\texttt{\textbackslash}Z_c(f))~is~compact\}$ can be established as parallel to the analogous properties of the ring $C_K(X)=\{f\in C(X):f~has~compact~support\}$ given in $4D$, \cite{GJ}.
\begin{theorem}
	Let $X$ be Hausdorff and zero-dimensional. Then : 
	\begin{enumerate}
		\item $C^c_K(X,F)\subset C_c(X,F)\cap C^*(X,F)$, where $C^*(X,F)=\{f\in C(X,F):cl_Ff(X)~is~compact\}$ and equality holds if and only if $X$ is compact.
		\item If $X$ is non-compact, then $C^c_K(X,F)$ is an ideal(proper) of $C_c(X,F)$.
		\item $C^c_K(X,F)$ is contained in every free ideal of $C_c(X,F)$. $C^c_K(X,F)$ itself is a free ideal of $C_c(X,F)$ if and only if $X$ is non-compact and locally compact.
		\item $X$ is nowhere locally compact if and only if $C^c_K(X,F)=\{0\}$ and this is the case when and only when $\beta_0X\texttt{\textbackslash}X$ is dense in $\beta_0X$. [Compare with $7F4$, \cite{GJ}] 
	\end{enumerate}
\end{theorem}
\begin{remark}
	$C^c_K(X,F)\subset \bigcap\{O^{p,F}_c:p\in \beta_0X\texttt{\textbackslash}X\}$.\\
	This follows from Theorem $3.1(5)$ and Theorem $3.2(3)$.
\end{remark}
To show that equality holds in the last inclusion relation, we need the following subsidiary result.
\begin{theorem}
	Let $f\in C_c(X,F)$ be such that $cl_{\beta_0X}Z_c(f)$ is a neighbourhood of $\beta_0X\texttt{\textbackslash}X$. Then $f\in C^c_K(X,F)$.
\end{theorem}
\begin{proof}
	It suffices to show that $supp(f)\equiv cl_X(X\texttt{\textbackslash}Z_c(f))$ is closed in $\beta_0X$ and hence compact. As $Z_c(f)$ is closed in $X$, it follows that $cl_{\beta_0X}Z_c(f)\cap (X\texttt{\textbackslash}Z_c(f))=\phi$. The hypothesis tells that there exists an open set $W$ in $\beta_0X$ such that $\beta_0X\texttt{\textbackslash}X\subset W\subset cl_{\beta_0X}Z_c(f)$. Hence $W\cap (X\texttt{\textbackslash}Z_c(f))=\phi$, which further implies because $W$ is open in $\beta_0X$ that $W\cap cl_{\beta_0X}(X\texttt{\textbackslash}Z_c(f))=\phi$. Consequently $W\cap cl_X(X\texttt{\textbackslash}Z_c(f))=\phi$. Since $\beta_0X\texttt{\textbackslash}X\subset W$, it follows therefore that no point of $\beta_0X\texttt{\textbackslash}X$ is a limit point of $cl_X(X\texttt{\textbackslash}Z_c(f))$ in the space $\beta_0X$. Thus there does not exist any limiting point of $cl_X(X\texttt{\textbackslash}Z_c(f))$ outside it in the entire space $\beta_0X$. Hence $cl_X(X\texttt{\textbackslash}Z_c(f))$ is closed in $\beta_0X$.
\end{proof}
\begin{theorem}
	Let $X$ be zero-dimension and Hausdorff. Then 
	$C^c_K(X,F)=\bigcap\{O^{p,F}_c:p\in \beta_0X\texttt{\textbackslash}X\}$
\end{theorem}
\begin{proof}
	Let $f\in O^{p,F}_c$ for each $p\in \beta_0X\texttt{\textbackslash}X$. Then $cl_{\beta_0X}Z_c(f)$ is a neighbourhood of each point of $\beta_0X\texttt{\textbackslash}X$ in the space $\beta_0X$. It follows from Theorem $3.4$ that $f\in C^c_K(X,F)$. Thus $\bigcap\{O^{p,F}_c:p\in \beta_0X\texttt{\textbackslash}X\}\subset C^c_K(X,F)$. The reversed implication relation is already realized in Remark $3.3$. Hence $C^c_K(X,F)=\bigcap\{O^{p,F}_c:p\in \beta_0X\texttt{\textbackslash}X\}$.
\end{proof}
\section{Von Neumann regularity of $C_c(X,F)$ versus $P-$space $X$}
We recall from \cite{SABM2004} that $X$ is called $P_F-$space if $C(X,F)$ is a Von-Neumann regular ring. By borrowing the terminology from \cite{RS2013}, we call a zero-dimensional space $X$, a countably $P_F-$space or $CP_F-$space if $C_c(X,F)$ is Von-Neumann regular ring. Thus in this terminology, $CP_\mathbb{R}-$spaces are precisely $CP-$spaces introduced in \cite{RS2013}, Definition $5.1$. It is still undecided whether there exist an ordered field $F$ and a zero-dimensional space $X$ for which $X$ is a $P_F-$space without being a $P-$space (see the comments preceding Definition $3.3$ in \cite{SABM2004}). However, we shall prove that subject to the restrictions imposed on the field $F$, already used several times in this paper, $CP_F-$spaces and $P-$spaces are one and the same. We want to mention in this context that the zero set of a function $f$ in $C(X,F)$ may not be a $G_{\delta}-$set [see Theorem $2.2$ in \cite{SABM2004}]. In contrast, we shall show that the zero set of a function lying in $C_c(X,F)$ is necessarily a $G_{\delta}-$set. Before proceeding further, we make the assumption throughout the rest of this article that the ordered field $F$ is either uncountable or a countable subfield of $\mathbb{R}$.
\begin{theorem}
	A zero set $Z\in Z_c(X,F)$ is a $G_{\delta}-$set.
\end{theorem}
\begin{proof}
	We can write $Z=Z_c(f)$ for some $f\geq 0$ in $C_c(X,F)$. Since $f(X)$ is a countable subset of $F$, we can write, $f(X)\texttt{\textbackslash}\{0\}=\{r_1,r_2,...,r_n,...\}$; a countable set in $F^+$. It follows that $Z_c(f)=\bigcap_{n=1}^\infty f^{-1}(-r_n,r_n)=$ a $G_{\delta}-$set in $X$.
\end{proof}
The following results are generalized versions of Proposition $4.3$, Theorem $5.5$ and Corollary $5.7$ in \cite{RS2013}.
\begin{theorem}
	If $A$ and $B$ are disjoint closed sets in $X$ with $A$, compact, then there exists $f\in C_c(X,F)$ such that $f(A)=\{0\}$ and $f(B)=\{1\}$
\end{theorem}
\begin{theorem}
	For $f\in C_c(X,F)$, $Z_c(f)$ is a countable intersection of clopen sets in $X$
\end{theorem}
\begin{proof}
	As in the proof of Theorem $4.1$, we can assume $f\geq 0$ and $f(X)\texttt{\textbackslash}\{0\}=\{r_1, r_2,...,r_n,...\}$. Now if $F$ is an uncountable ordered field, then for each $n\in\mathbb{N}$, we can choose $s_n\in F$ such that $0<s_n<r_n$ and $s_n\notin\{r_1, r_2,...,r_n,...\}$. On the other hand, if $F$ is a countable subfield of $\mathbb{R}$, then we can pick up for each $n\in\mathbb{N}$ an irrational point denoted by the same symbol $s_n$ with the above-mentioned condition, i.e., $0<s_n<r_n$ and $s_n\notin\{r_1, r_2,...,r_n,...\}$. It follows that $Z_c(f)=\bigcap_{n=1}^\infty f^{-1}(-s_n,s_n)=\bigcap_{n=1}^\infty f^{-1}[-s_n,s_n]=$ a countable intersection of clopen sets in $X$.
\end{proof}
\begin{theorem}
	A countable intersection of clopen sets in $X$ is a zero set in $Z_c(X,F)$ (equivalently, a countable union of clopen sets in $X$ is a co-zero set, i.e., the complement in $X$ of a zero set in $Z_c(X,F)$).
\end{theorem}
\begin{proof}
	Since in any topological space, a countable union of clopen sets can be expressed as a countable union of pairwise disjoint clopen sets, we can start with a countable family $\{L_i\}^\infty_{i=1}$ of pairwise disjoint clopen sets in $X$. For each $n\in\mathbb{N}$, define a function $e_n:X\to F$ as follows : $e_n(L_n)=\{1\}$ and $e_n(X\texttt{\textbackslash}L_n)=\{0\}$. Then $e_n\in C_c(X,F)$ and is an idempotent in this ring. Furthermore, it is easy to see that if $m\neq n$, then $e_m.e_n=0$. Let $h(x)=\sum_{n=1}^\infty\frac{e_n(x)}{3^n},~x\in X$. Then $h:X\to F$ is a continuous function and $h(X)\subset \{0, \frac{1}{3}, \frac{1}{3^2},...\}$. Thus $h\in C_c(X,F)$. It is clear that $\bigcup_{n=1}^\infty L_n=X\texttt{\textbackslash}Z_c(h)$.
\end{proof}
\begin{theorem}
	$Z_c(X,F)$ is closed under countable intersection.
\end{theorem}
\begin{proof}
	Follows from Theorem $4.3$ and Theorem $4.4$.
\end{proof}
\begin{theorem}
	Suppose a compact set $K$ in $X$ is contained in a $G_\delta-$set $G$. Then there exists a zero set $Z$ in $Z_c(X,F)$ such that $K\subset Z\subset G$.
\end{theorem}
\begin{proof}
	We can write $G=\bigcap_{n=1}^\infty W_n$ where each $W_n$ is open in $X$. For each $n\in\mathbb{N}$, $K$ and $X\texttt{\textbackslash}W_n$ are disjoint closed sets in $X$ with $K$ compact. Hence by Theorem $4.2$, there exists $g_n\in C_c(X,F)$ such that $g_n(K)=\{0\}$ and $g_n(X\texttt{\textbackslash}W_n)=\{1\}$. It follows that $K\subset Z_c(g_n)\subset W_n$ for each $n\in\mathbb{N}$. Consequently, $K\subset\bigcap_{n=1}^\infty Z_c(g_n)\subset G$. But by Theorem $4.5$, we can write $\bigcap_{n=1}^\infty Z_c(g_n)=Z_c(g)$ for some $g\in C_c(X,F)$. Hence $K\subset Z_c(g)\subset G$.
\end{proof}
Before giving several equivalent descriptions of the defining property of $CP_F-$space in the manner $4J$ of \cite{GJ} and the theorem $5.8$ in \cite{RS2013}. We like to introduce a suitable modified countable version of $\mathscr{m}-$topology on $C(X)$ as dealt with in $2N$, \cite{GJ}.\\
For each $g\in C_c(X,F)$ and a positive unit $u$ in this ring, set $M_F(g,u)=\{f\in C_c(X,F):|f(x)-g(x)|<u(x)~for~each~x\in X\}$. Then it can be proved by routine computation that $\{M_F(g,u):g\in C_c(X,F),~u,~a~\\positive~unit~in~C_c(X,F)\}$ is an open set for some topology on $C_c(X,F)$, which we call $\mathscr{m}_c^F-$topology on $C_c(X,F)$. A special case of this topology with $F=\mathbb{R}$ is already considered in \cite{QM2020}, Section $3$. The following two can be established by making straight forward modifications in the arguments adopted to prove Theorem $3.1$ and Theorem $3.7$ in \cite{QM2020}.
\begin{theorem}
	Each maximal ideal in $C_c(X,F)$ is closed in the $\mathscr{m}_c^F$-topology.
\end{theorem}
\begin{theorem}
	For any ideal $I$ in $C_c(X,F)$, its closure in $\mathscr{m}_c^F-$\\topology is given by : $\overline{I}=\bigcap\{M^{p,F}_c:p\in\beta_0X~and~M^{p,F}_c\supset I\}\equiv$ the intersection of all the maximal ideals in $C_c(X,F)$ which contains $I$.
\end{theorem}\vspace{-2mm}
[compare with $7Q2$, \cite{GJ}]
\begin{theorem}
	An ideal $I$ in $C_c(X,F)$ is closed in $\mathscr{m}_c^F-$topology if and only if it is the intersection of all the maximal ideals in this ring which contains $I$.
\end{theorem}\vspace{-2mm}
[This follows immediately from Theorem $4.8$]\\
We are now ready to offer a bunch of statements, each equivalent to the requirement that $X$ is a $CP_F-$space.
\begin{theorem}
	Let $X$ be a zero-dimensional Hausdorff space and $F$, a totally ordered field with the property mentioned in the beginning of this section. Then the following statements are equivalent : 
	\begin{enumerate}
		\item $X$ is a $CP_F-$space.
		\item Each zero set in $Z_c(X,F)$ is open.
		\item Each ideal in $C_c(X,F)$ is a $Z_{F_c}-$ideal.
		\item For all $f,g$ in $C_c(X,F)$, $<f,g>=<f^2+g^2>$.
		\item Each prime ideal in $C_c(X,F)$ is maximal.
		\item For each $p\in X$, $M^c_{p,F}=O^c_{p,F}$.
		\item For each $p\in\beta_0X$, $M^{p,F}_c=O^{p,F}_c$.
		\item Each ideal in $C_c(X,F)$ is the intersection of all the maximal ideals containing it.
		\item Each $G_\delta-$set in $X$ is open (which eventually tells that $X$ is a $P-$space)
		\item Every ideal in $C_c(X,F)$ is closed in the $\mathscr{m}^F_c-$topology. 
	\end{enumerate}
\end{theorem}
\begin{proof}
	Equivalence of the first eight statements can be proved by making an almost repetition of the arguments to prove the equivalence of the analogous statements in $4J$, \cite{GJ} [Also see the Theorem $5.8$ in \cite{RS2013}]. We prove the equivalence of the statements $(2), (9), (10)$.\\
	$(9)\implies(2)$ is immediate because of Theorem $4.1$.\\
	$(2)\implies(9)$ : Let $(2)$ be true and $G$ be a non-empty $G_\delta-$set in $X$. Then by Theorem $4.6$, for each point $x\in G$, there exists a zero set $Z_x\in Z_c(X,F)$ such that $x\in Z_x\subset G$. Since $Z_x$ is open in $X$ by $(2)$, it follows that $x$ is an interior point of $G$. In other words, $G$ is open in $X$.\\
	Equivalence of $(8)$ and $(10)$ follows from Theorem $4.9$.
\end{proof}
\begin{remark}
	On choosing $F=\mathbb{Q}$ in Theorem $4.10$, we get that a zero dimensional space $X$ is a $P-$space if and only if $C(X,\mathbb{Q})$ is a Von Neumann regular ring, i.e., $X$ is a $P_{\mathbb{Q}}-$space. Thus each $P_{\mathbb{Q}}-$space is a $P-$space. But we note that, though the cofinality character of $\mathbb{Q}$ is $\omega_0$, it is not Cauchy complete. This improves the conclusion of the Theorem $3.5$ in \cite{SABM2004}, which says that if $F$ is a Cauchy complete totally ordered field with cofinality character $\omega_0$, then every $P_F-$space is a $P-$space.
\end{remark}
Open question : If $p\in\beta_0X$, then does the set of prime ideals in $C_c(X,F)$ that lie between $O^{p,F}_c$ and $M^{p,F}_c$ make a chain?
\bibliographystyle{plain}

}
\end{document}